\documentclass[12pt]{article}

\usepackage{amsthm}
\usepackage{amscd}
\usepackage{amsfonts}
\usepackage{amssymb}
\usepackage{amsgen}
\usepackage{amsmath}
\usepackage{amsopn}
\usepackage{verbatim}
\usepackage{xypic}
\usepackage{xspace}
\usepackage{multicol}
\usepackage{url}
\usepackage{upref}

\theoremstyle{plain}
\newtheorem{thm}{Theorem}[section]
\newtheorem{lem}[thm]{Lemma}

\theoremstyle{definition}

\theoremstyle{remard}

 \font\cyr=wncyr10
 \newcommand{\nc}{\newcommand}

\DeclareMathOperator{\Sym}{{Sym}}

\nc{\per}[1]{\underset{#1}{\boldsymbol \pi}\,}

 \nc{\MT}{{\rm MT}}
 \nc{\XX}{{X}}
 \nc{\gF}{{\varPhi}}
 \nc{\ot}{\otimes}

 \nc{\tG}{\tilde{G}}
 \nc{\wht}{\widehat}
 \nc{\bwg}{{\bigwedge}}
 \nc{\wg}{{\wedge}}
 \nc{\mmu}{{\boldsymbol{\mu}}}
 \nc{\mal}{{{\scriptstyle \maltese}}}
 \nc{\fA}{{\mathfrak A}}
 \nc{\HH}{{\mathfrak H}}
 \nc{\ra}{\rightarrow}
 \nc{\ors}{{\bfs}}
 \nc{\orr}{{\bfr}}
 \nc{\os}{{\overset}}
 \nc{\G}{{\mathbb G}}
 \nc{\F}{{\mathbb F}}
 \nc{\Z}{{\mathbb Z}}
 \nc{\R}{{\mathbb R}}
 \nc{\N}{{\mathbb N}}
 \nc{\ZN}{{\mathbb Z_{\ge 0}}}
 \nc{\Q}{{\mathbb Q}}
 \nc{\C}{{\mathbb C}}
 \nc{\CP}{{\mathbb{CP}}}
 \nc{\Cnn}{{\mathbb C}_{\ge 0}}
 \nc{\Cp}{{\mathbb C}_{>0}}
 \nc{\MPV}{{\mathcal{MPV}}}
 
 \nc{\tB}{{\tilde B}}
\nc{\gemn}{{\mathfrak n}}
 \nc{\suf}{{\ast\,}}
 \nc{\sufq}{{\ast_q\,}}
 \nc{\gam}{{\gamma}}
 \nc{\gG}{{\Gamma}}
 \nc{\om}{{\omega}}
 \nc{\vep}{{\varepsilon}}
 \nc{\ga}{{\alpha}}
 \nc{\gl}{{\lambda}}
 \nc{\gb}{{\beta}}
 \nc{\gf}{{\varphi}}
 \nc{\gd}{{\delta}}
 \nc{\orgd}{{\vec \gd\,}}
 \nc{\gs}{{\sigma}}
 \nc{\gth}{{\theta}}
 \nc{\gS}{{\Sigma}}

 \nc{\gk}{{\kappa}}
  \nc{\gz}{{\zeta}}
 \nc{\tgz}{{\tilde{\zeta}}}
 \nc{\gO}{{\Omega}}
 \nc{\sif}{{\mathcal S}}
 \nc{\gt}{{\tau}}
 \nc{\Lra}{\Longrightarrow}
 \nc{\lra}{\longrightarrow}
 \nc{\lmaps}{\longmapsto}
 \nc{\fS}{{\mathfrak S}}
 \nc{\DD}{{\mathfrak D}}
 \nc{\Llra}{\Longleftrightarrow}
 \nc{\ol}{\overline}
 \nc{\ola}{\overleftarrow}
 \nc{\lms}{\longmapsto}
 \nc{\cv}{{{\mathsf c}{\mathsf v}}}
 \nc{\zq}{{\zeta_q}}
 \nc\qup{{q\uparrow 1}}
 \nc{\us}{\underset}
 \nc{\tn}{{\tilde{n}}}
 \nc{\gD}{{\Delta}}
 \nc{\bi}{{\bf i}}
 \nc{\bfone}{{\bf 1}}

 \nc{\bfa}{{\bf a}}
 \nc{\bfb}{{\bf b}}
 \nc{\bfc}{{\bf c}}
 \nc{\bfd}{{\bf d}}
 \nc{\bfe}{{\bf e}}
 \nc{\bff}{{\bf f}}
 \nc{\bfg}{{\bf g}}
 \nc{\bfi}{{\bf i}}
 \nc{\bfj}{{\bf j}}
\nc{\obfi}{{\overrightarrow{\boldsymbol \imath}}}
\nc{\obfj}{{\overrightarrow{\boldsymbol \jmath}}}
\nc{\obfd}{{\overrightarrow{\bf d}}}
\nc{\veps}{{\varepsilon}}
 \nc{\bfn}{{\bf n}}
 \nc{\bfl}{{\bf l}}
 \nc{\bfk}{{\bf k}}
 \nc{\bfm}{{\bf m}}
 \nc{\bfo}{{\bf o}}
 \nc{\bfp}{{\bf p}}
 \nc{\bfq}{{\bf q}}
 \nc{\bfr}{{\bf r}}
 \nc{\bfs}{{\bf s}}
 \nc{\bft}{{\bf t}}
 \nc{\bfu}{{\bf u}}
 \nc{\bfv}{{\bf v}}
 \nc{\bfw}{{\bf w}}
 \nc{\bfx}{{\bf x}}
 \nc{\bfy}{{\bf y}}
 \nc{\bfz}{{\bf z}}
 \nc{\bfB}{{\bf B}}
 \nc{\bfP}{{\bf P}}
 \nc{\bfQ}{{\bf Q}}
 \nc{\bfY}{{\bf Y}}
 \nc{\bfgb}{{\boldsymbol \gb}}
 \nc{\bfga}{{\boldsymbol \ga}}
 \nc{\bfrho}{{\boldsymbol \rho}}
 \nc{\bfchi}{{\boldsymbol \chi}}
 \nc{\QX}{{\Q\langle \bfX\rangle}}
 \nc{\QY}{{\Q\langle \bfY\rangle}}
 \nc{\CX}{{\C\langle \bfX\rangle}}
 \nc{\CY}{{\C\langle \bfY\rangle}}
 \nc{\QXX}{{\Q\langle\!\langle \bfX\rangle\!\rangle}}
 \nc{\QYY}{{\Q\langle\!\langle \bfY\rangle\!\rangle}}
 \nc{\CXX}{{\C\langle\!\langle \bfX\rangle\!\rangle}}
 \nc{\CYY}{{\C\langle\!\langle \bfY\rangle\!\rangle}}

 \nc{\bbA}{{\mathbb A}}
 \nc{\bbB}{{\mathbb B}}
 \nc{\bbC}{{\mathbb C}}
 \nc{\bbD}{{\mathbb D}}
 \nc{\bbE}{{\mathbb E}}
 \nc{\bbF}{{\mathbb F}}
 \nc{\bbG}{{\mathbb G}}
 \nc{\bbH}{{\mathbb H}}
 \nc{\bbI}{{\mathbb I}}
 \nc{\bbJ}{{\mathbb J}}
 \nc{\bbK}{{\mathbb K}}
 \nc{\bbL}{{\mathbb L}}
 \nc{\bbM}{{\mathbb M}}
 \nc{\bbN}{{\mathbb N}}
 \nc{\bbO}{{\mathbb O}}
 \nc{\bbP}{{\mathbb P}}
 \nc{\bbQ}{{\mathbb Q}}
 \nc{\bbR}{{\mathbb R}}
 \nc{\bbS}{{\mathbb S}}
 \nc{\bbT}{{\mathbb T}}
 \nc{\bbU}{{\mathbb U}}
 \nc{\bbV}{{\mathbb V}}
 \nc{\bbW}{{\mathbb W}}
 \nc{\bbX}{{\mathbb X}}
 \nc{\bbY}{{\mathbb Y}}
 \nc{\bbZ}{{\mathbb Z}}
 \nc{\bba}{{\mathbb a}}
 \nc{\bbb}{{\mathbb b}}
 \nc{\bbc}{{\mathbb c}}
 \nc{\bbd}{{\mathbb d}}
 \nc{\bbe}{{\mathbb e}}
 \nc{\bbf}{{\mathbb f}}
 \nc{\bbg}{{\mathbb g}}
 \nc{\bbh}{{\mathbb h}}
 \nc{\bbi}{{\mathbb i}}
 \nc{\bbk}{{\mathbb k}}
 \nc{\bbl}{{\mathbb l}}
 \nc{\bbm}{{\mathbb m}}
 \nc{\bbn}{{\mathbb n}}
 \nc{\bbo}{{\mathbb o}}
 \nc{\bbp}{{\mathbb p}}
 \nc{\bbq}{{\mathbb q}}
 \nc{\bbr}{{\mathbb r}}
 \nc{\bbs}{{\mathbb s}}
 \nc{\bbt}{{\mathbb t}}
 \nc{\bbu}{{\mathbb u}}
 \nc{\bbv}{{\mathbb v}}
 \nc{\bbw}{{\mathbb w}}
 \nc{\bbx}{{\mathbb x}}
 \nc{\bby}{{\mathbb y}}
 \nc{\bbz}{{\mathbb z}}

 \nc{\MZV}{{\mathcal{MZV}}}
 \nc{\calA}{{\mathcal A}}
 \nc{\calB}{{\mathcal B}}
 \nc{\calC}{{\mathcal C}}
 \nc{\calD}{{\mathcal D}}
 \nc{\calE}{{\mathcal E}}
 \nc{\calF}{{\mathcal F}}
 \nc{\calG}{{\mathcal G}}
 \nc{\calH}{{\mathcal H}}
 \nc{\calI}{{\mathcal I}}
 \nc{\calJ}{{\mathcal J}}
 \nc{\calK}{{\mathcal K}}
 \nc{\calL}{{\mathcal L}}
 \nc{\calM}{{\mathcal M}}
 \nc{\calN}{{\mathcal N}}
 \nc{\calO}{{\mathcal O}}
 \nc{\calP}{{\mathcal P}}
 \nc{\calQ}{{\mathcal Q}}
 \nc{\calR}{{\mathcal R}}
 \nc{\calS}{{\mathcal S}}
 \nc{\calT}{{\mathcal T}}
 \nc{\calU}{{\mathcal U}}
 \nc{\calV}{{\mathcal V}}
 \nc{\calW}{{\mathcal W}}
 \nc{\calX}{{\mathcal X}}
 \nc{\calY}{{\mathcal Y}}
 \nc{\calZ}{{\mathcal Z}}
  \nc{\cala}{{\mathcal a}}
 \nc{\calb}{{\mathcal b}}
 \nc{\calc}{{\mathcal c}}
 \nc{\cald}{{\mathcal d}}
 \nc{\cale}{{\mathcal e}}
 \nc{\calf}{{\mathcal f}}
 \nc{\calg}{{\mathcal g}}
 \nc{\calh}{{\mathcal h}}
 \nc{\cali}{{\mathcal i}}
 \nc{\calj}{{\mathcal j}}
 \nc{\calk}{{\mathcal k}}
 \nc{\call}{{\mathcal l}}
 \nc{\calm}{{\mathcal m}}
 \nc{\caln}{{\mathcal n}}
 \nc{\calo}{{\mathcal o}}
 \nc{\calp}{{\mathsf p}}
 \nc{\calq}{{\mathcal q}}
 \nc{\calr}{{\mathcal r}}
 \nc{\cals}{{\mathcal s}}
 \nc{\calt}{{\mathcal t}}
 \nc{\calu}{{\mathcal u}}
 \nc{\calv}{{\mathcal v}}
 \nc{\calw}{{\mathcal w}}
 \nc{\calx}{{\mathcal x}}
 \nc{\caly}{{\mathcal y}}
 \nc{\calz}{{\mathcal z}}

 \nc{\frakA}{{\mathfrak A}}
 \nc{\frakB}{{\mathfrak B}}
 \nc{\frakC}{{\mathfrak C}}
 \nc{\frakD}{{\mathfrak D}}
 \nc{\frakE}{{\mathfrak E}}
 \nc{\frakF}{{\mathfrak F}}
 \nc{\frakG}{{\mathfrak G}}
 \nc{\frakH}{{\mathfrak H}}
 \nc{\frakI}{{\mathfrak I}}
 \nc{\frakJ}{{\mathfrak J}}
 \nc{\frakK}{{\mathfrak K}}
 \nc{\frakL}{{\mathfrak L}}
 \nc{\frakM}{{\mathfrak M}}
 \nc{\frakN}{{\mathfrak N}}
 \nc{\frakO}{{\mathfrak O}}
 \nc{\frakP}{{\mathfrak P}}
 \nc{\frakQ}{{\mathfrak Q}}
 \nc{\frakR}{{\mathfrak R}}
 \nc{\frakS}{{\mathfrak S}}
 \nc{\frakT}{{\mathfrak T}}
 \nc{\frakU}{{\mathfrak U}}
 \nc{\frakV}{{\mathfrak V}}
 \nc{\frakW}{{\mathfrak W}}
 \nc{\frakX}{{\mathfrak X}}
 \nc{\frakY}{{\mathfrak Y}}
 \nc{\frakZ}{{\mathfrak Z}}
 \nc{\fraka}{{\mathfrak a}}
 \nc{\frakb}{{\mathfrak b}}
 \nc{\frakc}{{\mathfrak c}}
 \nc{\frakd}{{\mathfrak d}}
 \nc{\frake}{{\mathfrak e}}
 \nc{\frakf}{{\mathfrak f}}
 \nc{\frakg}{{\mathfrak g}}
 \nc{\frakh}{{\mathfrak h}}
 \nc{\fraki}{{\mathfrak i}}
 \nc{\frakj}{{\mathfrak j}}
 \nc{\frakk}{{\mathfrak k}}
 \nc{\frakl}{{\mathfrak l}}
 \nc{\frakm}{{\mathfrak m}}
 \nc{\frakn}{{\mathfrak n}}
 \nc{\frako}{{\mathfrak o}}
 \nc{\frakp}{{\mathfrak p}}
 \nc{\frakq}{{\mathfrak q}}
 \nc{\frakr}{{\mathfrak r}}
 \nc{\fraks}{{\mathfrak s}}
 \nc{\frakt}{{\mathfrak t}}
 \nc{\fraku}{{\mathfrak u}}
 \nc{\frakv}{{\mathfrak v}}
 \nc{\frakw}{{\mathfrak w}}
 \nc{\frakx}{{\mathfrak x}}
 \nc{\fraky}{{\mathfrak y}}
 \nc{\frakz}{{\mathfrak z}}
 \nc{\so}{{\mathfrak so}}
 \nc{\sa}{{\mbox{{\scriptsize \cyr x}}}}
 \nc{\slfour}{{\mathfrak sl}_4}
 \nc{\one}{{\bf 1}}
 \nc{\zero}{{\bf 0}}
 \nc{\Qxy}{\Q\langle x,y\rangle}

\nc{\pD}[2]{\frac{\partial{#1}}{\partial{#2}}}

\begin{document}
\title{Sum Formula of Multiple Hurwitz-Zeta Values}
\author{Jianqiang Zhao}
\date{}
\maketitle

\begin{center}
 {Department of Mathematics, Eckerd College, St. Petersburg, FL 33711}
\end{center}

\begin{abstract}
Let $s_1,\dots,s_d$ be $d$ positive integers and define the multiple $t$-values
of depth $d$ by
\begin{equation*}
    t(s_1,\dots,s_d)=\sum_{n_1>\cdots>n_d\ge 1}
    \frac{1}{(2n_1-1)^{s_1}\cdots(2n_d-1)^{s_d}},
\end{equation*}
which is equal to the multiple Hurwitz-zeta value
$2^{-w}\gz(s_1,\dots,s_d;-\frac12,\dots,-\frac12)$ where
$w=s_1+\cdots+s_d$ is called the weight.
For $d\le n$, let $T(2n,d)$ be the sum of all multiple $t$-values
with even arguments whose weight is $2n$ and whose
depth is $d$.  Recently Shen and Cai
gave formulas for $T(2n,d)$ for $d\le 5$ in terms of $t(2n)$, $t(2)t(2n-2)$
and $t(4)t(2n-4)$. In this short note we generalize Shen-Cai's results to
arbitrary depth by using the theory of symmetric functions
established by Hoffman.
\end{abstract}

\section{Introduction}
In recent years multiple zeta functions and many different variations and generalizations
have been studied intensively due to their close relations to other objects in a lot of diverse
branches of mathematics and physics. In particular, a large number of identities are
establishes between their special values. In \cite{ShenCai2012} Shen and Cai found a few very interesting
equations which are similar in nature to Euler's identity of double zeta values. They
gave formulas of the sum $E(2n,d)$ of multiple zeta values at even arguments of
fixed depth $d$ and weight $2n$, for $d\le 4$. These have been generalized to
arbitrary depth by Hoffman \cite{HEnk}. In \cite{ShenCai2011}
Shen and Cai turned to the following values
\begin{equation*}
    t(s_1,\dots,s_d)=\sum_{n_1>\cdots>n_d\ge 1}
    \frac{1}{(2n_1-1)^{s_1}\cdots(2n_d-1)^{s_d}},
\end{equation*}
which we call multiple $t$-values of depth $d$ in this note. It is clear that this is equal to
$2^{-w}\gz(s_1,\dots,s_d;-\frac12,\dots,-\frac12)$ where $w=s_1+\cdots+s_d$ is called the weight.
Put
\begin{equation*}
 T(2n,d)=\sum_{\substack{j_1+\cdots+j_d=n\\ j_1,\dots,j_d\ge 1}} t(2j_1,\dots,2j_d).
\end{equation*}
Using similar but more complicated ideas from \cite{ShenCai2012} Shen and Cai gave a
few sum formulas for $T(2n,d)$ for $d\le 5$ in \cite{ShenCai2011}. For example,
\begin{equation}\label{equ:Tn5}
 T(2n,5)= \frac{7}{128}t(2n)-\frac{3}{64}t(2)t(2n-2)+\frac{1}{320}t(4)t(2n-4).
\end{equation}

In this note, we shall generalize
these to arbitrary depth using ideas from \cite{HEnk} where Hoffman applied
the theory of symmetric functions to study the generating function of $E(2n,d)$.
It turns out that we need both Bernoulli numbers $B_j$ and Euler numbers $E_j$
defined by the following generating functions respectively:
\begin{equation} \label{equ:EulerNumber}
 \frac{x}{e^x-1}=\sum_{j=0}^\infty B_j\frac{x^j}{j!}, \qquad
 \sec x=\sum_{j=0}^\infty (-1)^{j}E_{2j} \frac{x^{2j}}{ (2j)!},
\end{equation}
and the Euler numbers $E_{2j+1}=0$ for all $j\ge 0$.

Our main results are the following theorems.
\begin{thm} \label{thm:T2nd}
For $d\le n$,
\begin{equation*}
    T(2n,d)= \sum_{j=0}^{\lfloor \frac{d-1}2\rfloor}
\frac{(-1)^j\pi^{2j}}{2^{2d-2} (2j)! d} \binom{2d-2j-2}{d-1} t(2n-2j),
\end{equation*}
where $t(2j)=2^{-2j}(2^{2j}-1)\gz(2j)$. Or, equivalently,
\begin{equation*}
T(2n,d)=\binom{2d-2}{d-1}\frac{t(2n)}{2^{2d-2}d} -\sum_{j=1}^{\lfloor \frac{d-1}2\rfloor}
\binom{2d-2j-2}{d-1}  \frac{t(2j)t(2n-2j)}{2^{2d-3}(2^{2j}-1)B_{2j} d} .
\end{equation*}
\end{thm}
The next three cases after \eqref{equ:Tn5} are
\begin{align*}
T(2n,6) &= \frac{21}{512}t(2n)-\frac{7}{192}t(2)t(2n-2)+\frac{1}{256}t(4)t(2n-4),\\
T(2n,7) &= \frac{33}{1024}\gz(2n)-\frac{15}{512}t(2)t(2n-2)
+\frac{1}{256}t(4)t(2n-4)-\frac{1}{21504}t(6)t(2n-6),\\
T(2n,8) &= \frac{429}{16384}t(2n)-\frac{99}{4096}t(2)t(2n-2)+\frac{15}{4096}t(4)t(2n-4)-\frac{1}{12288}t(6)t(2n-6).
\end{align*}

As we mentioned in the above the proof of Theorem~\ref{thm:T2nd} utilizes the
generating function of $T(2n,d)$ defined by
$$\Phi(u,v) =1+\sum_{n\ge d\ge 1} T(2n,d)u^n v^d$$
for which we have the following result.
\begin{thm} \label{thm:gfun}
We have
\begin{equation*}
\Phi(u,v)= \cos(\pi\sqrt{(1-v)u}/2) \sec(\pi\sqrt{u}/2) .
\end{equation*}
\end{thm}

The next theorem involves Euler numbers and
is more useful computationally when the difference between $n$ and $d$ is small.
\begin{thm} \label{thm:longform}
For $d\le n$ we have
\begin{equation}  \label{equ:longform}
T(2n,d)=\frac{(-1)^{n-d}\pi^{2n}}{4^n(2n)!}\sum_{\ell =0}^{n-d}
\binom{n-\ell}{d}\binom{2n}{2\ell}E_{2\ell} .
\end{equation}
\end{thm}

This work was started while the the author was visiting Taida Institute for
Mathematical Sciences at National Taiwan University in the summer
of 2012. He would like to thank Prof. Jing Yu and Chieh-Yu Chang for
encouragement and their interest in his work.

\section{Proof of Theorem \ref{thm:gfun} and Theorem \ref{thm:longform}}
We first recall some results on symmetric functions contained in \cite{HEnk,Macdonald} with
some slight modification.
Let $\Sym$ be the subring of $\Q[\![x_1,x_2,\dots]\!]$ consisting
of the formal power series of bounded degree that are invariant
under permutations of the $x_j$.
Define elements $e_j$, $h_j$, and $p_j$ in $\Sym$ by the generating functions
\begin{align*}
E(u)&=\sum_{j=0}^\infty e_ju^j=\prod_{j=1}^\infty (1+ux_j),\\
H(u)&=\sum_{j=0}^\infty h_ju^j=\prod_{j=1}^\infty \frac1{1-ux_j} = E(-u)^{-1},\\
P(u)&=\sum_{j=1}^\infty p_ju^{j-1}=\sum_{j=1}^\infty \frac{x_j}{1-ux_j}=
\frac{H'(u)}{H(u)} .
\end{align*}
Define a homomorphism $\frakT:\Sym\to\R$ such that $\frakT(x_j)=1/(2j-1)^2$ for all $j\ge 1$.
Hence for all $n\ge 1$
$$\frakT(p_n)=t(2n)=\sum_{j\ge 1}\frac{1}{(2j-1)^{2n}}.$$

First we need a simple lemma.
\begin{lem} \label{lem:Tnn}
For any positive integer $n$ let $\{2\}^n$ be the string $(2,\dots,2)$ with $2$
repeated $n$ times. Then we have
\begin{equation}\label{equ:Tnn}
 t(\{2\}^n)=\frac{\pi^{2n}}{4^n(2n)!}.
\end{equation}
\end{lem}
\begin{proof}
It is easy to see that
 \begin{align*}
1+\sum_{n=1}^\infty  t(\{2\}^n) x^n=&
\prod_{j=1}^{\infty} \left (1+ \frac{x}{(2j-1)^2} \right) \\
= & \prod_{j=1}^{\infty} \left (1+ \frac{x}{j^2} \right) /\prod_{j=1}^{\infty} \left (1+ \frac{x}{(2j)^2} \right) \\
=& \frac{\sinh(\pi \sqrt{x})}{\pi \sqrt{x}} \cdot \frac{\pi \sqrt{x}/2} {\sinh(\pi \sqrt{x}/2)}\\
=& \cosh(\pi \sqrt{x}/2) \\
=& \sum_{n=1}^\infty  \frac{\pi^{2n}x^n}{4^n(2n)!} .
 \end{align*}
This finishes the proof of the lemma.
\end{proof}

Now let
$N_{n,d}$ be the sum of all the monomial symmetric functions
corresponding to partitions of $n$ having length $d$.
Then clearly
$$\frakT(N_{n,d})=T(2n,d).$$
As in \cite{HEnk} we may define
\begin{equation*}
     \calF(u,v)=1+\sum_{n\ge d\ge 1}N_{n,d}u^nv^d ,
\end{equation*}
then $\frakT$ sends $\calF(u,v)$ to the generating function
$$\Phi(u,v) =1+\sum_{n\ge d\ge 1} T(2n,d)u^nv^d .$$
By Lemma \ref{lem:Tnn} we have
\begin{equation} \label{base}
\frakT(e_n)= t(\{2\}^n)=\frac{\pi^{2n}}{4^n(2n)!}.
\end{equation}
Hence
\begin{equation*}
\frakT(E(u))=  \cosh(\pi \sqrt{u}/2),
\end{equation*}
and
\begin{equation*}
\frakT(H(u))=\frakT(E(-u)^{-1})=1/\cosh(\pi \sqrt{-u}/2)=\sec(\pi \sqrt{u}/2).
\end{equation*}
Thus by \cite[Lemma 1]{HEnk} $\calF(u,v)=E((v-1)u)H(u)$ and we get
\begin{align*}
\Phi(u,v)=\frakT(E((v-1)u)H(u))
=&\cosh(\pi\sqrt{(v-1)u}/2)\sec(\pi\sqrt{u}/2) \\
=&\cos(\pi\sqrt{(1-v)u}/2)
\!\,\sec(\pi\sqrt{u}/2).
\end{align*}
This proves Theorem \ref{thm:gfun}.

Setting $v=1$ in Theorem \ref{thm:gfun} we obtain
\begin{equation*}
\Phi(u,1)=\sec(\pi\sqrt{u}/2).
\end{equation*}
This yields immediately the following identity by \eqref{equ:EulerNumber}
\begin{equation}  \label{equ:Thn}
\frakT(h_n)=\sum_{d=1}^n T(2n,d)
=\frac{(-1)^{n}E_{2n}\pi^{2n}}{4^n (2n)!} .
\end{equation}

Now by \cite[Lemma 2]{HEnk} we have
\begin{equation*}
N_{n,d}=\sum_{\ell =0}^{n-d}\binom{n-\ell}{d}(-1)^{n-d-\ell}h_\ell e_{n-\ell}.
\end{equation*}
Applying the homomorphism $\frakT$ and using equation \eqref{equ:Tnn} and
\eqref{equ:Thn} we get Theorem~\ref{thm:longform} immediately.

\section{Proof of Theorems \ref{thm:T2nd} and a combinatorial identity}
We now rewrite the generating function $\Phi(4u,v)$ as follows using Theorem \ref{thm:gfun}:
\begin{equation*}
\Phi(4u,v)= \sum_{d\ge 0} v^d \tG_d(u)=\sec(\pi\sqrt{u})\cos(\pi\sqrt{(1-v)u})
=\sec(\pi\sqrt{u})\sum_{j=0}^\infty \frac{\pi^{2j}}{(2j)!} (v-1)^j u^j .
\end{equation*}
Let $D$ be the differential operator with respect to $u$. Then
\begin{align*}
 \tG_d(u)=&(-1)^d \sec(\pi \sqrt{u})\sum_{j\ge d}\frac{(-1)^j\pi^{2j}u^j}{(2j)!}\binom{j}{d} \\
=& \sec(\pi \sqrt{u}) \cdot \frac{(-u)^d}{d!}\cdot
D^d  \sum_{j\ge d}\frac{(-1)^j\pi^{2j} u^j}{(2j)!}  \\
=& \sec(\pi \sqrt{u}) \cdot \frac{(-u)^d}{d!}\cdot
D^d  \cos(\pi \sqrt{u}) \\
=&- \frac{\pi^2}2 \sec(\pi \sqrt{u}) \cdot \frac{(-u)^d}{d!}\cdot
D^{d-1}  \frac{\sin(\pi \sqrt{u}) }{\pi \sqrt{u}} \\
=& \frac{\pi^2 u}{2d} \frac{\tan(\pi \sqrt{u})}{\pi \sqrt{u}} G_{d-1}(u)
\end{align*}
by \cite[(12)]{HEnk} (the definition of $G_k$ is defined on page 9). 
By \cite[Lemma 3]{HEnk} we have
\begin{align} \label{equ:tG1}
 \tG_d(u)= & -\frac{\pi^2 u}{2d} \sum_{j=0}^{\lfloor \frac{d-2}2\rfloor}\frac{(-4\pi^2u)^j}
{2^{2d-3}(2j+1)!}\binom{2d-2j-3}{d-1}    \\
+&\frac{\pi \sqrt{u}}{2d}  \tan(\pi \sqrt{u}) \sum_{j=0}^{\lfloor \frac{d-1}2\rfloor}\frac{(-4\pi^2u)^j}{2^{2d-2}(2j)!}
\binom{2d-2j-2}{d-1} \label{equ:tG2} \\
=&
\frac{\pi \sqrt{u}}{2d}  \tan(\pi \sqrt{u})
\sum_{j=0}^{\lfloor \frac{d-1}2\rfloor}\frac{(-4\pi^2u)^j}{2^{2d-2}(2j)!}
\binom{2d-2j-2}{d-1}
+\text{terms of degree $<d$}. \notag
\end{align}
It is well-dnown that
\begin{equation*}
\tan x= \sum_{m=1}^\infty \frac{(-1)^{m-1} 2^{2m} (2^{2m}-1) B_{2m} x^{2m-1}}{(2m)!}.
\end{equation*}
Hence
\begin{equation*}
\frac{\pi \sqrt{u}}{2}  \tan(\pi \sqrt{u})
= \sum_{m=1}^\infty   4^m   t(2m) u^m .
\end{equation*}
Therefore $T(2n,d)$ is the coefficient of $u^n$ in
$$ \tG_d(u/4)=\frac1d \sum_{m=2}^\infty    t(2m) u^m \sum_{j=0}^{\lfloor \frac{d-1}2\rfloor}\frac{(-\pi^2u)^j}{2^{2d-2}(2j)!}
\binom{2d-2j-2}{d-1}.$$
This implies Theorem~\ref{thm:T2nd} immediately. Notice that by comparing
Theorem~\ref{thm:T2nd} and Theorem~\ref{thm:longform} we get the following
identity of between Bernoulli numbers and Euler numbers.
\begin{thm}
For all $d\le n$
\begin{equation*}
 \sum_{j=0}^{\lfloor \frac{d-1}2\rfloor}
\frac{ (2^{2n-2j}-1)B_{2n-2j} }{2^{2d-1}  d} \binom{2d-2j-2}{d-1} \binom{2n}{2j}
=\frac{(-1)^{n-d}\pi^{2n}}{4^n(2n)!}\sum_{\ell =0}^{n-d}
\binom{n-\ell}{d}\binom{2n}{2\ell}E_{2\ell}.
\end{equation*}
Further we have
\begin{align*}
\ & \sum_{j=0}^{\lfloor \frac{d-1}2\rfloor}
\frac{ (2^{2n-2j}-1)B_{2n-2j} }{2^{2d-1}  d} \binom{2d-2j-2}{d-1} \binom{2n}{2j} \\
=& \left\{
   \begin{array}{ll}
     0, & \hbox{if $n<d<2n$;} \\
     \displaystyle  \frac{n}{2^{2d-1} d}  \binom{2d-2n-1}{d-1} , & \hbox{if $d\ge 2n$.}
   \end{array}
 \right.
\end{align*}
\end{thm}
\begin{proof}
We only need to show the second identity.
Notice that when $d>n$ the coefficient of $u^nv^d$ is 0 in $\Phi(u,v)$. Thus the coefficient
of $u^n$ in $\tG_d(u/4)$ is zero. By \eqref{equ:tG1} and \eqref{equ:tG2} we have
\begin{align*}
& \sum_{j=0}^{\lfloor \frac{d-1}2\rfloor}
\frac{ (2^{2n-2j}-1)B_{2n-2j} }{2^{2d-1}  d} \binom{2d-2j-2}{d-1} \binom{2n}{2j} \\
=& \frac{(-1)^n (2n)!}{(2\pi)^{2n}} \times \text{Coeff.\ of $u^n$ of  
\eqref{equ:tG1} (i.e.\ $j=n-1$)} \\
=&\left\{
   \begin{array}{ll}
     0, & \hbox{$n<d<2n$;} \\
     \displaystyle  \frac{n}{2^{2d-1} d}  \binom{2d-2n-1}{d-1} , & \hbox{$d\ge 2n$,}
   \end{array}
 \right.
\end{align*}
as desired.
\end{proof}

\end{document}